\documentclass[12pt]{amsart}
\usepackage{amsmath,amssymb,amsbsy,amsfonts,amsthm,latexsym,
                       amsopn,amstext,amsxtra,euscript,amscd}
\usepackage{fullpage,url}

\newtheorem{thm}{Theorem}
\newtheorem{prop}[thm]{Proposition}
\newtheorem{cor}[thm]{Corollary}
\newtheorem{lem}[thm]{Lemma}

\newtheorem{rmk}[thm]{Remark}

\def\phi{\varphi}

\def\\{\cr}
\def\({\left(}
\def\){\right)}
\def\[{\left[}
\def\]{\right]}
\def\<{\langle}
\def\>{\rangle}

\def\cA{\mathcal A}

\def\cC{\mathcal C}

\def\S{\mathcal{S}}

\def\Z{\mathbb{Z}}

\begin{document}

\title[Squarefree smooth numbers and Euclidean prime generators]{Squarefree smooth numbers\\ and Euclidean prime generators}
\author{Andrew R. Booker}
\author{Carl Pomerance}
\address{School of Mathematics\\
Bristol University\\
University Walk\\
Bristol, BS8 1TW, UK}
\email{andrew.booker@bristol.ac.uk}
\address{Mathematics Department\\
Dartmouth College\\
Hanover, NH 03755, USA}
\email{carl.pomerance@dartmouth.edu}


\date{\today}

\pagenumbering{arabic}

\date{\today}

\begin{abstract}
We show that for each prime $p>7$, every residue mod $p$ can be represented by a squarefree
number with largest prime factor at most $p$.  We give two applications to recursive prime
generators akin to the one Euclid used to prove the infinitude of primes.
\end{abstract}

\maketitle

\section{Introduction}
\label{sec:1}
In \cite{mullin}, Mullin considered the sequence $\{p_k\}_{k=1}^\infty$
defined so that, for every $k\ge0$, $p_{k+1}$ is the smallest prime
factor of $1+p_1\cdots p_k$. From the argument employed by Euclid to
prove the infinitude of prime numbers, it follows that the $p_k$ are
pairwise distinct, and Mullin's sequence can thus be viewed as an explicit,
constructive form of the proof. A natural question, which Mullin posed,
is whether \emph{every} prime eventually occurs in the sequence.
Despite clear heuristic and empirical evidence that the answer must be yes,
it appears to be very difficult to prove anything substantial to that
end.\footnote{At least one of the authors thinks that Mullin's
question is likely undecidable.}

With this setting in mind, in \cite[Section 1.1.3]{cp} and \cite{b},
we (independently) described two variations of Euclid's argument
that allow greater flexibility and lead to sequences that provably
contain every prime number. We recall these constructions in detail in
Section~\ref{sec:nothing}. The main focus of this article is the following
question, which arises naturally as an ingredient in both
constructions, but is possibly of independent interest:
\begin{center}
\emph{For primes $p$, are all residue classes mod $p$ represented by\\
the positive integers that are both squarefree and $p$-smooth?}
\end{center}
(Recall that an integer $n$ is called \emph{$y$-smooth} if every prime
divisor of $n$ is $\le y$.)  Since there are $2^{\pi(p)}$ squarefree,
$p$-smooth, positive integers and only $p$ residue classes mod $p$, one
heuristically expects the answer to be yes, at least for large $p$.
(However, note that $y=p$ is best possible, since the zero residue class
mod $p$ is not attained by a $y$-smooth number for any $y<p$.) We will
show that, with two exceptions, this is indeed the case:
\begin{thm}\label{thm:main}
Let $p$ be a prime different from $5$ and $7$, and $a\in\Z$. Then there
is a squarefree, $p$-smooth, positive integer $n$ such that
$n\equiv a\pmod{p}$.
\end{thm}

The proof of Theorem~\ref{thm:main} consists of three largely independent
steps that we carry out in Sections~\ref{sec:char}--\ref{sec:smallp}.
Our key tools include a numerically explicit form of the P\'olya--Vinogradov
inequality, see Frolenkov and Soundararajan \cite{FS}, and a combinatorial result of Lev \cite{L} on $h$-fold sums of
dense sets.
In Section~\ref{sec:euclid} we apply Theorem~\ref{thm:main} to variants of
Euclid's argument, and thus show how one can generate all of the primes out
of nothing.  In the final section we mention a few related unsolved problems.

Related results on prime generators of Euclid type may be found in
\cite{bold}, \cite{PT}, \cite{Wa}, \cite{W}, and the references of
those papers.

\section{Large $p$: character sums}\label{sec:char}

For a prime $p$ and a positive integer $d\mid p-1$, let 
$$
H_{d,p}=\{h\in(\Z/p\Z)^*:~h^{(p-1)/d}~\equiv~1\pmod{p}\}
$$
denote the subgroup of $(\Z/p\Z)^*$ of index $d$.

\begin{prop}
\label{prop:cosets}
Let $p>3\times10^8$ be a prime and suppose $d\mid p-1$ with $d<\log p+1$. 
For each nonzero residue $m\pmod p$ there is some squarefree number
$j<p$ with $j\in mH_{d,p}$.
\end{prop}
\begin{proof}
We may assume that $d>1$, since otherwise we can take $j=1$.
Let $\chi$ be a character mod $p$ of order $d$.  Since $\chi^i$ for
$i=1,\dots,d$ runs over all of the characters mod $p$ of order
dividing $d$, we have that
\begin{equation}
\label{eq:charsum}
\frac1d\sum_{i=1}^d\sum_{j<p}\mu^2(j)\chi^{i}(j)\overline{\chi}^i(m)
\end{equation}
is the number of squarefree numbers $j<p$ with $\chi(j)=\chi(m)$, and so is the
number of squarefree numbers smaller than $p$ in the coset $mH_{d,p}$.
The principal character (the term when $i=d$) contributes
$$
\frac1d\sum_{j<p}\mu^2(j)
$$
to the sum.  This expression is $\sim\frac{6}{\pi^2}p/d$ as $p\to\infty$.  One
can get an explicit lower bound valid for all $p$ via the Schnirelmann
density of the squarefree numbers,
see \cite{R}. 
Thus,
\begin{equation}
\label{eq:main}
\frac1d\sum_{j<p}\mu^2(j)\ge\frac{53}{88d}(p-1).
\end{equation}

Our task is then to show that the other terms in \eqref{eq:charsum} are small
in comparison.
By recognizing a squarefree number by an inclusion-exclusion over square
divisors, we have for $1\le i<d$,
$$
\sum_{j<p}\mu^2(j)\chi^i(j)=\sum_{v\ge1}\mu(v)\chi^i(v^2)\sum_{j<p/v^2}\chi^i(j).
$$
We may discard the term $v=1$ since it is 0.  For $v>\frac12p^{1/4}+1$, we may use the
trivial estimate
$$
\sum_{v>\frac12p^{1/4}+1}\left|\sum_{j<p/v^2}\chi^i(j)\right|\le
\sum_{v>\frac12p^{1/4}+1}\,\sum_{j<p/v^2}1<p\sum_{v>\frac12p^{1/4}+1}\frac1{v^2}<2p^{3/4}.
$$
For $v\le \frac12p^{1/4}+1$ we use an explicit form of the P\'olya--Vinogradov inequality,
see \cite{FS}, where we may divide the estimate for even characters by 2
since our character sum is over an initial interval.
This gives
$$
\sum_{\substack{v>1\\v\le \frac12p^{1/4}+1}}\left|\sum_{j<p/v^2}\chi^i(j)\right|
\le\sum_{\substack{v>1\\v\le \frac12p^{1/4}+1}}\left(\frac1{2\pi}p^{1/2}\log p
+p^{1/2}\right)
\le \frac1{4\pi}p^{3/4}\log p+\frac12p^{3/4}.
$$
Thus,
$$
\left|\sum_{j<p}\mu^2(j)\chi^i(j)\right|\le p^{3/4}\left(\frac1{4\pi}\log p+\frac52\right).
$$
Hence
$$
\frac1d\sum_{i=1}^{d-1}\left|\sum_{j<p}\mu^2(j)\chi^i(j)\overline{\chi}^i(m)\right|
\le\left(1-\frac1d\right)p^{3/4}\left(\frac1{4\pi}\log p+\frac52\right),
$$
and we would like this expression to be smaller than the one in \eqref{eq:main}.
That is, we would like the inequality
$$
\frac{53(p-1)}{88d}>
\left(1-\frac1d\right)p^{3/4}\left(\frac1{4\pi}\log p+\frac52\right),
$$
to hold true, or equivalently,
$$
\frac{53(p-1)}{88p^{3/4}}>(d-1)\left(\frac1{4\pi}\log p+\frac52\right).
$$
Using $d<\log p+1$, we see that this inequality holds for all
$p>3\times10^8$.
\end{proof}
\begin{rmk}
{\rm Instead of \cite{FS} for our estimate of the character sum, we might
have used \cite{P} or we might have used the ``smoothed" version in \cite{LPS}.
The former would require raising the lower limit $3\times10^8$ slightly,
while the latter would likely lead to a reduction in the lower limit,
but at the expense of a more complicated proof.}
\end{rmk}

For a prime $p>3\times10^8$ and an integer $d\mid p-1$ with $d<\log p+1$,
let $\cC_{d,p}$ denote a set of squarefree coset representatives
for $H_{d,p}$ smaller than $p$ as guaranteed to exist by Proposition
\ref{prop:cosets}.  Also, let $\S_{d,p}$ denote the set of primes
that divide some member of $\cC_{d,p}$ and let $\S_p$ be the union
of all of the sets $\S_{d,p}$ for $d\mid p-1$, $d<\log p+1$.

Let $\omega(n)$ denote the number of distinct prime divisors of $n$.
It is known that $\omega(n)\le(1+o(1))\log n/\log\log n$ as $n\to\infty$.
We have the weaker, but explicit inequality: $\omega(n)<\log n$ for $n>6$.
To see this, note that it is true for $\omega(n)\le2$, since it holds
for $n=7$, and for $n\ge8$ we have $\log n>2$.  If $\omega(n)=k\ge3$, then
$n\ge6\cdot5^{k-2}$, so that $k\le(\log n+\log(25/6))/\log 5$, which
is smaller than $\log n$ for $n\ge11$.  But $k\ge3$ implies $n\ge30$.

As a corollary, we conclude that under the hypotheses of Proposition
\ref{prop:cosets}, we have each $\#\S_{d,p}<d\log p$ and $\#\S_p<\frac12(\log p+1)^3$.

\section{Proof of Theorem~\ref{thm:main} for large $p$}
Assume the prime $p$ exceeds $3\times10^8$. 
Let $\S_p$ denote the set of primes identified at
the end of the last section, let $N=p-1$, and let 
$$
K=\pi(N)-\#\S_p>\pi(N)-\frac12(\log p+1)^3
$$ 
denote the number of remaining primes smaller than $p$.

For an integer $m$ with $0<m<p$, let $f(m)$ denote
the number of unordered pairs of distinct primes $q,r$ with $q,r<p$,
$qr\equiv m\pmod p$ and $q,r\not\in\S_p$.
Set 
$$
\cA=\{m\in(0,p):f(m)>K/\sqrt{p}\}\cup\{1\}, \quad A=\#\cA.  
$$
\begin{lem}
\label{lem:Aest}
For $p>3\times 10^8$, we have
$$
A>\frac{N}{\log N}+2.
$$
\end{lem}
\begin{proof}
Evidently,
\begin{equation}
\label{eq:ftot}
\sum_{m=1}^{p-1}f(m)=\binom{K}2=\frac12K(K-1).
\end{equation}
Further, if two pairs $q,r$ and $q',r'$ are counted by $f(m)$, then either they have
no prime in common or they are the same pair.  Thus, for each $m$,
\begin{equation}
\label{eq:fmax}
f(m)\le \frac12K.
\end{equation}
Since
$$
\sum_{m\not\in\cA}f(m)\le(N-A)K/\sqrt{p},
$$
we have by \eqref{eq:ftot} that
$$
\sum_{m\in\cA}f(m)\ge\frac12K(K-1)-(N-A)K/\sqrt{p}.
$$
Thus, from \eqref{eq:fmax},
$$
A\ge\frac{1}{K/2}\sum_{m\in\cA}f(m)\ge K-1-2(N-A)/\sqrt{p}
>K-2\sqrt{p}-1.
$$
Since $K>\pi(N)-\frac12(\log p+1)^3$, by
using inequality (3.1) in \cite{RS} and $p>3\times10^8$, we have
the inequality in the lemma.
\end{proof} 

For a positive integer $k$, let $\cA^k$ denote the set of $k$-fold
products of members of $\cA$.
\begin{lem}
\label{lem:lev}
There are positive integers $d<\log p+1$, $k< 2\log p+3$ such that
$\cA^k$ contains a 
subgroup $H_{d,p}$ of $(\Z/p\Z)^*$.
\end{lem}
\begin{proof}
Let $g$ be a primitive root modulo $p$ and let $\cA'$ denote the set of
discrete logarithms of members of $\cA$ to the base $g$.  That is, $j\in\cA'$
with $0\le j<N$ if and only if $g^j\pmod p\in\cA$.
We now apply a theorem of Lev \cite[Theorem $2'$]{L} to the set $\cA'$.
With $\kappa:=\lceil(N-1)/(A-2)\rceil$, this
result implies that 
there are positive integers $d'\le\kappa$ and $k\le 2\kappa+1$ such
that $k\cA'$ contains $N$ consecutive multiples of $d'$.  Here, $k\cA'$
denotes the set of integers that can be written as the sum of $k$
members of $\cA'$.  Thus, reducing mod $N$, the set $k\cA'$ contains
a subgroup of $\Z/N\Z$ of index $d:=(d',N)$.  Hence, $\cA^k$ contains the
subgroup $H_{d,p}$ of $(\Z/p\Z)^*$.  Further, from Lemma \ref{lem:Aest},
we have $d\le d'\le\kappa< \log p+1$, which completes the proof.
\end{proof}

\begin{lem}
\label{lem:Hdp}
For the subgroup $H_{d,p}$ of $(\Z/p\Z)^*$ produced in Lemma \ref{lem:lev},
each member of $H_{d,p}$ has a representation modulo $p$ as a squarefree
number involving primes smaller than $p$ and not in $\S_{p}$ (and so not
in $\S_{d,p}$).
\end{lem}
\begin{proof}
Suppose that $m\in\cA^k$, so that 
\begin{equation}
\label{eq:prod}
m=m_1m_2\dots m_k\equiv (q_1r_1)(q_2r_2)\dots(q_kr_k)\pmod p,
\end{equation}
where each $m_i\in\cA$ and $m_i\equiv q_ir_i\pmod p$.  
This last product over primes is $p$-smooth, but is not
necessarily squarefree.  However, each $m_i\in\cA$ 
has many representations as $q_ir_i$, in fact at least $K/\sqrt{p}$
representations, with each representation involving two new primes.
So, if $k$ is small enough, there will be a representation of each
$m_i$ so that the product of primes in \eqref{eq:prod} is indeed squarefree. 
Now $k<2\log p+3$, so having at least $2(k-1)+1<4\log p+5$
representations for each member of $\cA$ is sufficient.  The number
of representations exceeds $K/\sqrt{p}$ and by the same calculation
that gave us the last step in 
Lemma \ref{lem:Aest}, we have this expression exceeding $\sqrt{p}/\log p$.
This easily exceeds $4\log p+5$ for $p>3\times10^8$.
\end{proof}

It is now immediate that every residue class mod $p$ contains a squarefree
number with prime factors at most $p$.  Indeed this is true for $0\pmod p$
--- take $p$ as the representative.  For a nonzero class $j\pmod p$, find
that member $m$ of $\cC_{d,p}$ with $j\in mH_{d,p}$, and write $j=mh$
with $h\in H_{d,p}$.  We have seen that each member of $H_{d,p}$ 
has a squarefree representative using primes smaller than $p$
and not in $\S_{d,p}$.  Since $m$ is squarefree and uses only primes in
$\S_{d,p}$ it follows that $mh$ is also squarefree using only primes smaller
than $p$.  This completes the proof of Theorem~\ref{thm:main} for
primes $p>3\times10^8$.

\section{Verification of Theorem~\ref{thm:main} for small $p$}\label{sec:smallp}
It remains only to verify the theorem for $p<3\times10^8$.  For $p>10^4$
we use the following simple strategy: Compute a primitive root
$g\pmod{p}$, and find pairwise coprime, squarefree, $p$-smooth numbers
$m_i$ such that $m_i\equiv g^{2^i}\pmod{p}$ for each nonnegative
integer $i\le\log_2(p-2)$. If this is possible then, given any
nonzero residue $n\pmod{p}$, we have $n\equiv g^k\pmod{p}$ for some
integer $k\in[0,p-2]$. Expressing $k$ in binary, viz.\ $k=\sum_{0\le
i\le\log_2(p-2)}b_i2^i$ for $b_i\in\{0,1\}$, we have $n\equiv\prod_{0\le
i\le\log_2(p-2)}m_i^{b_i}\pmod{p}$.  Since the $m_i$ are pairwise coprime,
the residue class of $n$ is thus represented by a squarefree $p$-smooth
number, as desired.

It is convenient to choose $m_i$ of the form $q_ir_i$ for primes
$q_i$, $r_i$. To find these efficiently, for each $i=0,1,2,\ldots$ we
search through small primes $q$, compute the smallest positive $r\equiv
q^{-1}g^{2^i}\pmod{p}$, test whether $r$ is prime, and ensure that $qr$
is coprime to $m_j$ for $j<i$. The only essential ingredient needed
to carry this out is a fast primality test; we used a strong Fermat
test to base $2$ coupled with the classification \cite{F} of small
strong pseudoprimes, which would allow us, in principle, to handle
any $p<2^{64}$.  Heuristically, one can expect this method to succeed
using $O(\log^3{p})$ arithmetic operations on numbers of size $p$, and
we found it to be very fast in practice; it takes just minutes to verify
the theorem for all $p\in(10^4,3\times10^8)$ on a modern multicore processor.

For $p<10^4$ we fall back on a brute-force algorithm: For each
integer $a\in[1,p-1]$, consider each of the numbers
$a,a+p,a+2p,\ldots$ until encountering one that divides
$\prod_{\substack{q\text{ prime}\\q<p}}q$.
This takes only seconds to check for all $p<10^4$ other than $5$ and
$7$. (For $p\in\{5,7\}$ one can see directly that $4+p\Z$ is not
represented, but all other residue classes are.)

\section{Generating all of the primes from nothing}\label{sec:nothing}
\label{sec:euclid}
We give two applications of Theorem~\ref{thm:main} to
Euclidean prime generators. The first was described without proof in
\cite[Section 1.1.3]{cp}; we supply the short proof here.
\begin{cor}\label{cor:d+1}
Starting from the emptyset, recursively define a sequence of primes,
where if $n$ is the product of the primes generated so far, take
$$
p=\min\{q\hbox{ prime}:q\nmid n,~q\mid d+1 \hbox{ for some }d\mid n\}
$$
as the next prime.
This sequence begins with $2,3,7,5$, and then produces the
primes in order.
\end{cor}
\begin{proof}
Let $p\ge11$ be the least prime not yet produced after the 4th step
and let $n$ be the product of the primes smaller than $p$.
By Theorem \ref{thm:main}, there exists some $d\mid n$ with
$d\equiv-1\pmod p$.  Thus, we generate $p$ at the next step.
\end{proof}

The second variant was described in \cite{b}. With
Theorem~\ref{thm:main} in hand, we can give a shorter proof.
(As will be clear from the proof, there is an obstruction preventing the
terms from appearing in strict numerical order in this case,
so the conclusion is weaker than that of Corollary~\ref{cor:d+1}.)
\begin{cor}
\label{cor:d+d'}
Starting from the empty set, recursively generate a sequence of
primes where if $n$ is the product of the primes generated so far,
take as the next prime some prime factor of some $d+n/d$, where $d\mid n$.
Such a sequence can be chosen to contain every prime.
\end{cor}
\begin{proof}
One such sequence begins with $2$, $3$, $5$, $13$, $7$. Suppose
$p>7$ is a prime number and that the sequence constructed so far
contains every prime smaller than $p$ and perhaps some primes larger
than $p$, but it does not contain $p$.  Let $n$ be the product of
the primes generated so far.  
If $\left(\frac{-n}{p}\right)=1$
then there exists $a\in(\Z/p\Z)^*$ such that $a+n/a=0$. By
Theorem~\ref{thm:main} there exists $d\mid n$ belonging to
the class of $a$, so we can choose $p$ as the next prime.  On the other hand,
since $p>5$,
if $\left(\frac{-n}{p}\right)=-1$, then 
\cite[Lemma~3(i)]{b} guarantees the existence of $a\in(\Z/p\Z)^*$
such that $\left(\frac{a+n/a}{p}\right)=-1$.  By
Theorem~\ref{thm:main} there exists $d\mid n$ belonging to the
class of $a$, and by multiplicativity it follows that $d+n/d$
has a prime factor $q$ satisfying $\left(\frac{q}{p}\right)=-1$. Choosing
$q$ as the next prime, we thus have $\left(\frac{-qn}{p}\right)=1$,
so by the above argument we can now take $p$.  This completes the proof.
\end{proof}

\section{Comments and problems}
The prime generator of Corollary \ref{cor:d+1} has its roots in a construction
in the primality test of \cite{APR}. There, a finite initial set of primes
is given with product $I$, and then one takes the product $E$ of all primes
$p\nmid I$ with $p\mid d+1$ for some $d\mid I$.  The primality test can be
used for numbers $n<E^2$ and runs in time $I^{O(1)}$.  Further, it is shown
that there are choices of $I,E$ with $I=(\log E)^{O(\log\log\log E)}$, so
the test runs in ``almost" polynomial time.  The same $I,E$ construction 
(with $I$ no longer required to be squarefree) is
used in the finite fields primality test of Lenstra \cite{Len}.

In Theorem \ref{thm:main} we insist that the squarefree $p$-smooth integers used
be positive.  If negatives are allowed, then the primes 5 and 7 are no longer
exceptional cases.  Further, if ``$d$" is allowed to be negative in the context of
the prime generator in Corollary \ref{cor:d+1}, the primes are generated in order. 
(For this to be nontrivial, $d$ should not be chosen as $-1$.)

Suppose we use the generator of Corollary \ref{cor:d+d'} by always returning the least
prime possible, and say this sequence of primes is $q_1,q_2,\dots$.
Does $\{q_k\}$ contain every prime?  Is there way of choosing the sequence
in Corollary \ref{cor:d+d'} such that every prime is generated and the $k$th 
prime generated is asymptotically equal to the $k$th prime?
Is it true that any sequence containing all primes as in Corollary \ref{cor:d+d'}
cannot contain the primes in order starting from some point?
These questions might all
be asked if we allow prime factors of $d\pm n/d$ in Corollary \ref{cor:d+d'} instead of
just $d+n/d$.

Presumably in Theorem \ref{thm:main}, when $p$ is large, residues $a\pmod p$
have many representations as squarefree $p$-smooth
integers.  Say we try to minimize the largest squarefree $p$-smooth used.
For $p>7$, let  $M(p)$ be the smallest number such that every residue mod~$p$
can be represented by a squarefree $p$-smooth number at most $M(p)$.
Our proof shows that $M(p)\le p^{O(\log p)}$.  We conjecture that $M(p)\le p^{O(1)}$.

We mentioned in the Introduction that the condition ``$p$-smooth" in
Theorem \ref{thm:main} cannot be relaxed to ``$y$-smooth" for any
$y<p$, since otherwise the residue class $0\pmod p$ will not be
represented.  However, we may ask for the smallest number $y=y(p)$
such that every nonzero residue class mod~$p$ can be represented
by a $y$-smooth squarefree number.  Via the Burgess inequality, it is likely that one
can show that $y(p)\le p^{1/(4\sqrt{e})+o(1)}$ as $p\to\infty$.    
Assuming the Generalized Riemann Hypothesis for Kummerian fields
(as Hooley \cite{hooley} did in his GRH-conditional proof of Artin's conjecture), it is
likely that one can prove that $y(p)=O((\log p)^2)$.
We note that if one drops the ``squarefree'' condition then these
statements follow from work of Harman \cite{harman}
unconditionally and Ankeny \cite{ankeny} under GRH; see also the recent
paper \cite{LLS} for a strong, numerically explicit version of the
latter.

\section*{Acknowledgments}
We thank Enrique Trevi\~no and Kannan Soundararajan for some helpful
comments.  
The first-named author was partially supported by EPSRC Grant {\tt EP/K034383/1}.

\end{document}